\documentclass[11pt,reqno,final]{article}

\usepackage{amsmath,amsfonts,amssymb,amsthm,version}
\usepackage{mathrsfs,fancybox,pifont}
\usepackage{graphicx}
\usepackage{url,hyperref}
\usepackage[notcite,notref]{showkeys}
\usepackage{color}
\usepackage{subfigure,multirow}
\usepackage{epstopdf}
\usepackage{cases}
\usepackage{mathtools}
\usepackage{algorithm,algorithmic}
\usepackage{authblk}
\usepackage{lipsum}

\allowdisplaybreaks

\numberwithin{equation}{section}

\theoremstyle{plain}
\newtheorem{exam}{Example}[section]
\newtheorem{theorem}[exam]{Theorem}
\newtheorem{lemma}[exam]{Lemma}
\newtheorem{remark}[exam]{Remark}

\newtheorem{definition}[exam]{Definition}

\begin{document}

\date{}

\title{Stable invariant manifold  for generalized ODEs   with applications  \\ to
  measure differential equations
\footnote{This paper was jointly supported by the Natural Science Foundation of Zhejiang Province (No. LZ24A010006),
and the National Natural Science Foundation of China (No. 11671176, 11931016).
}
}
\author
{
Weijie Lu$^{a}$\,\, \,\,
Paolo Piccione$^{b}$\,\,\,\,\,
Yonghui Xia$^{c}\footnote{Corresponding author. yhxia@zjnu.cn. }$
\\
{\small \textit{$^a$ School of Mathematics  Science,  Zhejiang Normal University, 321004, Jinhua, China}}\\
{\small \textit{$^b$ Departamento de Matem\'{a}tica, Universidade de S\~{a}o Paulo, S\~{a}o Paulo, Brazil}} \\
{\small \textit{ $^b$ School of Mathematics, Foshan University, 528000, Foshan, China}}\\
{\small { Email: luwj@zjnu.edu.cn; piccione@ime.usp.br; xiadoc@163.com; yhxia@zjnu.cn } }
}

\maketitle

\begin{abstract}
This paper establishes the stable invariant manifold for a new kind of differential equations defined by Kurzweil integral, so-called {\em generalized ODEs} on a Banach space.
The nonlinear generalized ODEs are formulated as
$$
\frac{dz}{d\tau}=D[\Lambda(t)z+F(z,t)],
$$
where $\Lambda(t)$ is a bounded linear operator on  a Banach space $\mathscr{Z}$ and
$F(z,t)$ is a nonlinear Kurzweil integrable function on $\mathscr{Z}$.
The letter $D$ represents that   generalized ODEs are defined via its solution,
and   $\frac{dz}{d\tau}$  only a notation. Hence, generalized ODEs are  fundamentally a notational representation of a class of integral equations.
 Due to the differences between the theory of generalized ODEs and ODEs,  it is difficult to extended  the stable manifold  theorem of ODEs to generalized  ODEs.
In order to overcome the   difficulty, we  establish a generalized Lyapunov-Perron equation in the frame of  Kurzweil  integral theory.
Subsequently, we present a stable invariant manifold theorem for nonlinear generalized ODEs when their linear parts exhibit an exponential dichotomy.
As effective applications, we finally derive results concerning the existence of stable manifold  for measure differential equations and impulsive differential equations.

\end{abstract}

\section{Introduction}
\subsection{Background on the generalized ODEs }

%  The concept of generalized ordinary differential equations (for short, generalized ODEs)  was introduced by Kurzweil (\cite{K-CMJ1957,K-CMJ1958}) in 1957, with the aim of relaxing some classical results regarding the continuous dependence on parameters of solutions of ODEs.
% Kurzweil   proposed the notion of the generalized Perron integral, which is now known as the Kurzweil-Henstock integral.
% This integral is equivalent to the restricted Denjoy integral and proved to be effective for handling highly oscillating functions and functions with multiple discontinuities.
%For example, consider a function $\gamma:[0,1]\to \mathbb{R}$ defined as $\gamma(t)=\sin\frac{1}{t}$ for $t\in(0,1]$ and $\gamma(0)=0$. This function exhibits bounded oscillation. Similarly, the function $\rho:[0,1]\to \mathbb{R}$ with $\rho(t)=\frac{1}{t}(\sin\frac{1}{t}+\cos\frac{1}{t})$ for $t\in(0,1]$ and $\rho(0)=0$ has an unbounded oscillation point. The Kurzweil-Henstock integral is particularly suitable for analyzing such functions.

%Hence, several classical theorems and tools typically applied in traditional ODEs are no longer applicable in the context of generalized ODEs. For example, the chain rule and the multiplication rule of derivatives, as well as differential and integral mean value theorems, among others, may not hold.
%

\quad  Recently, the generalized ordinary differential equations (for short, generalized ODEs) on  a Banach space have attracted extensive attention,  which has the form of:
\begin{equation}\label{1F}
  \frac{dz}{d\tau}=D [K(z,t)],
\end{equation}
where $K:\mathscr{Z}\times \mathbb{R}\to \mathscr{Z}$ is a given function on a Banach space $\mathscr{Z}$.
It was first  introduced by Kurzweil (\cite{K-CMJ1957,K-CMJ1958}) in 1957, with the aim of relaxing some classical results regarding the continuous dependence on parameters of solutions of ODEs.
Notice that by Schwabik (\cite[Remark 3.2]{S-BOOK1992}),
the letter $D$ of \eqref{1F} means that \eqref{1F} is a generalized differential equation, this concept being defined via its solution,
and  $\frac{dz}{d\tau}$ is only a notation.
Hence, generalized ODEs are essentially a notational representation of a class of integral equations.
Indeed, recall an example from Schwabik's book (see \cite[pp. 100]{S-BOOK1992}): we let $K(z,t)=f(t)$, where
 $f:[0,1]\to \mathbb{R}$ is  a continuous function whose derivative does not exist at any point in $[0,1]$.
Then,  from the definition of Kurzweil   integral (refer to Definition \ref{KW}), we see that
\[
\int_a^b DK(z(\tau),\eta)=\int_a^b Df(\eta)=f(b)-f(a), \quad \forall a,b  \in [0,1],
\]
which means that $z:[0,1]\to\mathbb{R}$ given by $z(\eta)=f(\eta)$ for $\eta\in [0,1]$ is a solution of
$\frac{dz}{d\tau}=DK(z,t)=Df(t)$, but it has no derivative at any point in the interval $[0,1]$.
This example precisely shows that generalized ODEs are only a notational representation of  integral equations.
Moreover, Schwabik in \cite{S-BOOK1992}  pointed out that  generalized ODEs are   formal equation-like objecst for which one has defined its solutions.
As a result, the study of generalized ODEs presents an intriguing and challenging research avenue.
For more comprehensive insights and detailed information, readers are encouraged to refer to the monograph of Bonotto-Federson-Mesquita (\cite{BFM-BOOK2021})
and the references therein.

 The conceptual frameworks and fundamental theory of generalized ODEs were established by  Schwabik (\cite{S-MB1996,S-MB1999,S-BOOK1992}).
 Later, Collegari et al. (\cite{CFF-CMJ}) proved the variation of constants formula for generalized ODEs;
 Bonotto et al.  (\cite{BFS-JDE2018,BFS-JDDE2020}) respectively presented the concept of exponential dichotomy and its robustness for linear generalized ODEs;
 Afonso et al. (\cite{ABFG-MN2012}) derived the     boundedness of solutions for generalized ODEs;
  Federson et al. (\cite{FGMT-N2022}) obtained the existence of periodic solutions for autonomous generalized ODEs;
  Gallegos et al. (\cite{Ga-JDE21}) first explored the exponential stability of Lyapunov functions under the framework of generalized  ODEs, and then 
 Andrade da Silva et al. (\cite{AMT-JDE2022}) proved the converse Lyapunov theorem for generalized ODEs;
  Bonotto et al. (\cite{BFM-JGA2023}) investigated the boundary value problems for generalized ODEs;
  Gallegos and Robledo (\cite{Ga-Pro24})  established the stability results 
 for generalized linear ODEs and explicitly linked them to the concept of variational stability introduced by S. Schwabik.
 Recently,   following the fundamental  works by Bonotto et al.  (\cite{BFS-JDE2018,BFS-JDDE2020}), 
 Gallegos and Robledo (\cite{Ga-JDDE24}) presented a criterion for exponential dichotomy of periodic generalized linear ODEs and gave an application to admissibility; Lu and Xia (\cite{LX-JDE}) established a
Hartman-Grobman type linearization theorem for nonlinear generalized ODEs.
 On the other hand, generalized ODEs encompass various types of other differential equations as special cases,
such as the classical ODEs, measure differential equations (MDEs), dynamic equations on time scales, functional differential equations (FDEs), as well as impulsive differential equations (IDEs).
There is a substantial amount of research in these fields, and we do not attempt to provide a detailed list.
However, a pretty application of generalized ODEs is MDEs, which have well studied, see e.g.,
Piccoli (\cite{P-ARMA}) established the fundamental theory of MDEs;
Piccoli and Rossi (\cite{P-DCDS})  studied the measure dynamics with probability vector fields and sources;
 Federson et al. (\cite{FMS-JDE2012}) clarified  the relation between of MDEs and generalized ODEs, and
obtained some results on the existence and uniqueness of solutions, continuous dependence  and periodic averaging for MDEs;
Meng and Zhang (\cite{MG-JDE,MG-PAMS}) presented the dependence of solutions and eigenvalues of MDEs, and claimed the  extremal problems of eigenvalues for MEDs;
Wen and Zhang (\cite{Wen-DCDSB,Wen-JDE}) were devoted to study the principal eigenvalues problem of MDEs;
Chu et al. \cite{Chu-ADV,Chu-MA} presented the sharp bounds   for Dirichlet eigenvalues problem.
Another application of generalized ODEs is IDEs, see e.g., the work of  Federson and Schwabik \cite{FS-DIE2006}, Afonso et al. \cite{ABFS-JDE2011} and the references therein.

\subsection{Motivations and contributions of this paper}

\quad
 The theory of stable invariant manifold, which was initiated in Hadamard \cite{Hadamard} and Perron \cite{Perron} independently,
is one of the most fundamental results in dynamical systems.
The  invariant manifold theory  has been studied by many scholars in a variety of dynamical systems, such as,  invariant manifolds  on infinite-dimensional spaces   (\cite{Chow-Lu-1}),
manifolds on complex dynamics (\cite{FS-MA1977}),
stable manifolds of holomorphic hyperbolic maps (\cite{FS-IJM2004}),
stable manifolds with generalized exponential dichotomy (\cite{ZWN-Chinese}), invariant manifolds without gap condition (\cite{ZWM-ADM})),
inertial manifolds for PDEs (\cite{Mallet-Paret-Sell}),
 stable manifolds for nonuniform hyperbolicity (\cite{JDDE-Bento}),
 invariant manifolds for random dynamical systems (\cite{Shen-SCM}),
 invariant manifolds for noninstantaneous impulsive systems (\cite{Li-Wang1,Li-Wang2}), and others.

Observe that the works on the stable manifolds involve the hyperbolicity of linear systems.
  While exponential dichotomy plays a crucial role in generalizing the hyperbolicity of linear autonomous systems to linear nonautonomous systems.
   Bonotto, Federson and Santos \cite{BFS-JDE2018} presented the concept of exponential dichotomy for the linear generalized ODEs
\begin{equation}\label{1}
  \frac{dz}{d\tau}=D[\Lambda(t)z],
\end{equation}
on a Banach space $\mathscr{Z}$,   where $\Lambda(t)$ is a bounded linear operator on $\mathscr{Z}$.
 They  gave some sufficient and necessary conditions for Eq. \eqref{1} to ensure that \eqref{1} admits an exponential dichotomy.

Due to the  differences between the generalized ODEs and classical ODEs, a very natural question arises:
 does a stable invariant manifold exist under the framework of generalized ODEs?
 In this paper, we consider the stable invariant manifold for a class of generalized ODEs.
Our main aim is to establish the stable invariant manifold for sufficiently small perturbations of an exponential dichotomy.
  Furthermore, we prove that any solutions of generalized ODEs outside the stable manifold are unbounded.
  In particular, we obtain a generalized Lyapunov-Perron equation of the form (see Lemma \ref{lemma1} below)
\begin{align*}
  z(t)=& \mathscr{V}(t,s)P(s)z(s)+\int_s^t DF(z(\tau),\gamma)-\int_{s}^t d_{\sigma}[\mathscr{V}(t,\sigma)P(\sigma)]\int_s^\sigma DF(z(\tau),\gamma) \\
  &+ \int_{t}^\infty d_{\sigma}[\mathscr{V}(t,\sigma)(Id-P(\sigma))] \int_s^\sigma DF(z(\tau),\gamma),\quad \mathrm{for} \; t\geq s,
\end{align*}
where $\mathscr{V}(t,s)$ is the fundamental operator of Eq. \eqref{1} and $P:\mathbb{R}\to \mathscr{Z}$ is a projection.
    We remark that unlike the integral form of ODEs, IDEs and other differential equations,
the presence of the term $\int_s^t DF(z(\tau),\gamma)$ at the right-hand side of this generalized Lyapunov-Perron  integral increases the difficulties of the proof (see Lemmas \ref{lemma1} and \ref{lemma3} below).
  As applications, we derive results on the existence of stable manifold for measure differential equations (MDEs) and impulsive differential equations (IDEs).

\subsection{Outline of this paper}

\quad    The rest of this paper is organized as follows.
   In Section 2, we present some basic theory of generalized ODEs,
which includes regulated functions,  bounded variation functions, Kurzweil integrals and some result on the  existence and  uniqueness  of solutions and dichotomy theory, as well as others.
  In Section 3, we focus on the stable manifold for generalized ODEs. Theorem \ref{theorem1} gives our main result of this paper.
  Lemma \ref{lemma1} reports the existence of the generalized Lyapunov-Perron type integral equation for nonlinear generalized ODEs and obtains an inverse result.
  Lemma \ref{lemma3} presents the uniqueness of the solution for the Lyapunov-Perron integral due to the contraction fixed point theory.
  Finally, we prove the existence, invariance of the stable manifold, and show  unboundedness of the solutions outside the stable manifold.
  In Section 4, we give the existence of stable manifolds for two classes of differential equations.
  The results on the stable manifold for a class of MDEs and IDEs are described in Sections 4.1, 4.2, respectively.
  These results are obtained by using the correspondence between generalized ODEs and MDEs (or IDEs).

\section{Basic notations and concepts for generalized ODEs}
\subsection{Regulated functions,   bounded variation functions  and Kurzweil integral}

 \quad  Let $(\mathscr{Z},\|\cdot\|)$ be the Banach space. Given $c,d\in\mathbb{R}$ and $c<d$,   we say that $g:[c,d]\to \mathscr{Z}$ is a
{\em regulated function} if  the  limits exist:
\[
g(t^-)=\lim\limits_{r \to t^-} g(r), \; t\in(c,d] \quad  \mathrm{and} \quad g(t^+)= \lim\limits_{r\to t^+} g(r), \; t\in[c,d).
\]

Define
\[
\mathscr{G}([c,d],\mathscr{Z})=\{g:[c,d]\to \mathscr{Z}| g \;\mathrm{ is} \; \mathrm{a} \; \mathrm{regulated}  \; \mathrm{function} \;
\mathrm{with}  \; \sup_{t\in[c,d]}\|g(t)\|<\infty\}
\]
and $\|g\|_0:=\sup\|g(t)\|$, $t\in[c,d]$. Then $(\mathscr{G}([c,d],\mathscr{Z}),\|\cdot\|_0)$ is a Banach space
(see \cite{Honig-book}).

 We say that a finite set $D=\{t_0,t_1,\cdots, t_j\}\subset [c,d]$ satisfying $c=t_0\leq t_1\leq \cdots\leq t_j=d$ is   a $division$ of $[c,d]$.
If $|D|$ is the number of subintervals of $[t_{j-1},t_j]$ of a division $D$ of $[c,d]$, then we write $D=\{t_0,\cdots,t_{|D|}\}$.
 Denote  by $\mathscr{D}[c,d]$ the set of all division of $[c,d]$.
We say that   $g:[c,d]\to\mathscr{Z}$ is   a {\em variation function}  if
\[
  \mathrm{var}_{c}^{d} g:=\sup_{D\in\mathscr{D}[c,d]}\sum_{j=1}^{|D|}\|g(t_j)-g(t_{j-1})\|.
\]
  If $\mathrm{var}_{c}^{d} g<\infty$, then $g$ is a   bounded variation function   on  $[c,d]$.

  Define
\[
\mathscr{BVF}([c,d],\mathscr{Z})=\{g\in\mathscr{Z}|g \; \mathrm{is} \; \mathrm{a} \; \mathrm{bounded} \; \mathrm{variation} \; \mathrm{function}\; \mathrm{with} \; \|g(c)\|+\mathrm{var}_{c}^d g <\infty\}
\]
and $\|g\|_{\mathscr{BVF}}:=\|g(c)\|+\mathrm{var}_{c}^d g$.
  Then $(\mathscr{BVF}([c,d],\mathscr{Z}),\|\cdot\|_{\mathscr{BVF}})$ is a Banach space.
  Furthermore, $\mathscr{BVF}([c,d],\mathscr{Z}) \subset \mathscr{G}([c,d],\mathscr{Z})$ (see \cite{Honig-book}).

Next, we  recall the  notion of Kurzweil integral, which is defined by \cite{K-CMJ1957,S-BOOK1992}.
We say that a tagged division of  $[c,d]$ is a finite set of point-interval pairs $D=\{(s_j,[t_{j-1},t_j]): j=1,2,\cdots,|D|\}$,
 where $c=t_0\leq t_1 \leq \cdots \leq t_{|D|}=d$ is a division of $[c,d]$ and $s_j \in [t_{j-1},t_j]$ is the $tag$ of  $[t_{j-1},t_j]$.
A $gauge$ on $[c,d]$ is an arbitrary  positive function $\delta:[c,d]\to (0,\infty)$.
A tagged division $D=\{(s_j,[t_{j-1},t_j]): j=1,2,\cdots,|D|\}$ of $[c,d]$ is $\delta$-fine for any gauge $\delta$ on $[c,d]$ if
\[
[t_{j-1},t_j]\subset (s_j-\delta(s_j),s_j+\delta(s_j)).
\]

\begin{definition}\label{KW}(Kurzweil integrable)
We say that
 a  function $V:[c,d]\times [c,d]\to \mathscr{Z}$ is   {\em Kurzweil integrable} on $[c,d]$  if there exists a unique element $J\in \mathscr{Z}$ satisfying for any $\epsilon>0$,
there is a gauge $\delta$ of $[c,d]$ such that for each $\delta$-fine tagged division $D=\{(s_j,[t_{j-1},t_j]), j=1,2,\cdots,|D| \}$ of $[c,d]$,
we have
\[\|K(V,d)-J\|<\epsilon,\]
 where $K(V,d)=\sum_{j=1}^{|D|} [V(s_j,t_j)-V(s_j,t_{j-1})]$, and we write $J=\int_c^d DV(s,t)$ in this case.
\end{definition}

\subsection{Fundamental theory for the linear generalized ODEs}

\quad  Let $\mathscr{B}(\mathscr{Z})$ be the space of all bounded linear operators on a  Banach space $(\mathscr{Z},\|\cdot\|)$.
  Consider the linear generalized ODEs
\begin{equation}\label{LGODE}
  \frac{dz}{d\tau}=D[\Lambda(t)z]
\end{equation}
on $\mathscr{Z}$, where  $\Lambda:\mathbb{R}\to \mathscr{B}(\mathscr{Z})$. As point out in  \cite{S-BOOK1992}, a function $z:[c,d]\to\mathscr{Z}$
is called  a solution of \eqref{LGODE} if and only if
\begin{equation}\label{Z1}
   z(d)=z(c)+\int_{c}^{d} D[\Lambda(s)z(\tau)].
\end{equation}
  The integral on the right-hand side of \eqref{Z1} is a Kurzweil integral,
  which is denoted by the Perron-Stieltjes integral $\int_c^d d[\Lambda(s)]z(s)$ (\cite{S-BOOK1992,S-MB1996}), since we represent $\int_c^d D[\Lambda(t)z(\tau)]$ as Riemann-Stieltjes sum taking the form of
$\sum [\Lambda(t_j)-\Lambda(t_{j-1}]z(\tau_j)$.

   Given an interval $I\subset \mathbb{R}$, throughout this paper we assume that:
\\
(A1)  $\Lambda\in \mathscr{BVF}([c,d],\mathscr{B}(\mathscr{Z}))$ for  $[c,d]\subset I$;
\\
(A2) $(Id+[\Lambda(t^+)-\Lambda(t)])^{-1}\in \mathscr{B}(\mathscr{Z})$ for $t\in I \backslash \sup(I)$ and
$(Id-[\Lambda(t)-\Lambda(t^-)])^{-1}\in \mathscr{B}(\mathscr{Z})$ for  $t\in I \backslash \inf(I)$,
where $\mathrm{Id}$ is an identity operator, $\Lambda(t^+)=\lim\limits_{r\to t^+}\Lambda(r)$ and $\Lambda(t^-)=\lim\limits_{r\to t^-}\Lambda(r)$.

 The above assumptions guarantee the existence and the uniqueness for the solution of \eqref{LGODE}.

\begin{lemma}(\cite{S-MB1999}, Theorem 2.10)
Assume that (A1) and (A2) hold.
Given $(t_0,z_0)\in I\times \mathscr{Z}$.
Then the following linear GODE
\begin{equation}\label{Z2}
  \begin{cases}
  \frac{dz}{d\tau}=D[\Lambda(t)z(\tau)],\\
  z(t_0)=z_0,
  \end{cases}
\end{equation}
admits a unique solution on $I$.
\end{lemma}

\begin{lemma}\label{ppp-lemma}(\cite{CFF-CMJ}, Theorem 4.3)
We say that an operator $V:I\times I\to \mathscr{B}(\mathscr{Z})$ is  a fundamental solution operator of \eqref{LGODE} if
\begin{equation*}
  V(t,s)=Id+\int_s^t d[\Lambda(r)]V(r,s), \quad t,s\in I,
\end{equation*}
and for any fixed $s\in I$,  $V(\cdot,s)$ is  locally bounded variation in $I$.
  Furthermore, $z(t)=V(t,s)z_s$ is a unique solution of \eqref{Z2}. %given by

\end{lemma}

\begin{lemma}(\cite{CFF-CMJ}, Theorem 4.4)
  The operator $V:I\times I \to \mathscr{B}(\mathscr{Z})$ has the following properties:\\
(1) $V(t,t)=Id$;\\
(2) for any $[c,d]\subset I$, there is a positive constant $C_v>0$ satisfying
\[\begin{split}
&\|V(t,s)\|\leq C_v, \quad t,s\in[c,d], \quad \mathrm{var}_c^d V(t,\cdot)\leq C_v,\quad t\in[c,d],\\
&\mathrm{var}_c^d V(\cdot,s)\leq C_v,\quad s\in[c,d];
\end{split}\]
(3) for any $t,r,s\in I$, $V(t,s)=V(t,r)V(r,s)$;\\
(4) $V^{-1}(t,s)\in\mathscr{B}(\mathscr{Z})$ and $V^{-1}(t,s)=V(s,t)$.
\end{lemma}

\begin{definition}(Exponential dichotomy   \cite{BFS-JDE2018})
    We say that  linear generalized ODEs \eqref{LGODE} admit an exponential dichotomy on $I$, if there exist a   projection $P:I \to \mathscr{Z}$ and constants $K,\alpha>0$
satisfying
\begin{equation}\label{ED}
  \begin{cases}
   \| \mathscr{V}(t)(Id-P(t))\mathscr{V}^{-1}(s)\|\leq Ke^{\alpha(t-s)} \quad \mathrm{for} \; t<  s,\\
    \|\mathscr{V}(t)P(t)\mathscr{V}^{-1}(s)\|\leq Ke^{-\alpha(t-s)} \quad \mathrm{for} \; t\geq  s,
  \end{cases}
\end{equation}
where $\mathscr{V}(t):=V(t,0)$ and $\mathscr{V}^{-1}(t):=V(0,t)$.
\end{definition}

Next, we   define the stable and unstable spaces at time $s$ by
\[
\mathscr{E}^s(s)=P(s)\mathscr{Z} \quad \mathrm{and} \quad \mathscr{E}^u (s)=(Id-P(s))\mathscr{Z}.
\]
  In order to understand this concept,  we give  an example  related on exponential dichotomy for the generalized ODEs.\\
{\bf Example.} Consider a linear   generalized ODE
\begin{equation}\label{Exam-1}
  \frac{dz}{d\tau}=D[a(t)z],
\end{equation}
where $z\in\mathbb{R}$, $t\in I$ and
\[
a(t):=
\begin{cases}
 0, \quad t\in J:=\{t_i\}_{i=-n}^{n}\subset I, n\in\mathbb{N}, \\
 t,\quad t\in I\backslash J.
\end{cases}
\]
It is observed that the function $a(t)$ in this context is a scalar function with a countable number of discontinuities.
It is evident that assumptions (A1) and (A2) are satisfied. Thus, we conclude that Eq. \eqref{Exam-1} exhibits an exponential expansion, representing a specific instance of exponential dichotomy.
In fact, Lemma \ref{ppp-lemma} establishes that
\[
\|\mathscr{V}(t)\|=\|V(t,0)\|\leq 1+\left\|\int_0^t V(r,0)d[a(r)]\right\|, \quad t\geq 0.
\]
By using the Gronwall-type inequality (\cite{BFS-JDDE2020}, Corollary 1.43), one has
\[
\|\mathscr{V}(t)\| \leq e^{var_0^t a}\leq e^t.
\]
Letting $P(t)=O$, then
\[
 \| \mathscr{V}(t)(Id-P(t))\mathscr{V}^{-1}(s)\|\leq e^{t-s}, \quad \forall t\leq s, \; t,s\in I,
\]
which indicates that Eq. \eqref{Exam-1} has an expansion. Similarly, we can construct an example such that $\frac{dz}{d\tau}=D[a(t)z]$ admits a contraction as well.
%\\
%{\bf Example 2.} Let $a(t)$ in Example 1 be defined as follows:
%\[
%a(t):=
%\begin{cases}
%\frac{|\sin t|}{t}, \quad t>0, t\in I, \\
% 0, \quad t=0.
%\end{cases}
%\]
%It is not difficult to show that $\frac{dz}{d\tau}=D[a(t)z]$ has a contraction. In fact,
%\[
%\|\mathscr{V}(t)\|=\|V(t,0)\|\leq 1+\left\|\int_0^t V(r,0)d[a(r)]\right\| \leq e^{var_0^ t a}\leq e^{\frac{|\sin t|}{t}}, \quad t\geq 0.
%\]
%For any $t,s\in I$ with $t\geq s\geq 0$, there exists a constant $\alpha>0$ such that
%\[
%\frac{|\sin t|}{t}-\frac{|\sin s|}{s} \leq -\alpha(t-s).
%\]
%Then putting $P(t)=Id$, we get
%\[
% \| \mathscr{V}(t)P(t)\mathscr{V}^{-1}(s)\|\leq e^{-\alpha (t-s)}, \quad t\geq s\geq 0, \; t,s\in I.
%\]
%Therefore,  the required result is verified.

\subsection{The perturbation theory of the linear generalized ODEs}
\begin{lemma}\label{changshu}(\cite{CFF-CMJ}, Theorem 4.10)
  Assume that  (A1) and (A2) hold. %{\color{blue}(These assumptions have not appeared yet, see page~\pageref{NL})}
  If $F:\mathscr{Z} \times [c,d]\to \mathscr{Z}$ is  Kurzweil integrable,
  $[\tilde{c},\tilde{d}]\subseteq [c,d]$
and $t_0\in [\tilde{c},\tilde{d}]$, then the generalized ODEs %$z:[\tilde{c},\tilde{d}]\to \mathscr{Z}$ is a solution of
\begin{equation*}
  \begin{cases}
   \frac{dz}{d\tau}=D[\Lambda(t)z+F(z,t)],\\
   z(t_0)=z_0,
  \end{cases}
\end{equation*}
are equivalent to the following integral equations
%iff it is a solution of the following integral equation
\[
z(t)=V(t,t_0)z_0+\int_{t_0}^tDF(z(\tau),\gamma)-\int_{t_0}^t d_{\sigma}[V(t,\sigma)]\left(\int_{t_0}^{\sigma}DF(z(\tau),\gamma) \right).
\]
%where $V(\cdot,\cdot)$ is given by \eqref{EP}.
\end{lemma}

We then define  a special class of functions $F:\mathscr{Z} \times I\to \mathscr{Z}$. For convenience, we write $\Omega:=\mathscr{Z} \times I$.
\begin{definition}\label{defclassF}
  Given a function $h:I \to \mathbb{R}$ is nondecreasing.
  We say that $F:\Omega\to \mathscr{Z}$ belongs to the class $\mathscr{F}(\Omega,h)$ if
\begin{equation}\label{HHH1}
  \|F(z,t_2)-F(z,t_1)\|\leq |h(t_2)-h(t_1)|
\end{equation}
for any $(z,t_2)$, $(z,t_1)\in\Omega$ and
\begin{equation}\label{HHH2}
  \|F(z,t_2)-F(z,t_1)-F(w,t_2)+F(w,t_1)\|\leq \|z-w\||h(t_2)-h(t_1)|
\end{equation}
for any $(z,t_2)$, $(z,t_1)$,  $(w,t_2)$ $(w,t_2)\in\Omega$.
\end{definition}

\begin{lemma}\cite{BFS-JDE2018}
Suppose that the linear generalized ODEs \eqref{LGODE} satisfy (A1)--(A2) and admit an
exponential dichotomy.   If $g\in \mathscr{G}(\mathbb{R},\mathscr{Z})$ and the Perron-Stieltjes integrals
 \begin{equation}\label{V3}
   \int_{-\infty}^t d_{\sigma}[\mathscr{V}(t)P(t)\mathscr{V}^{-1}(\sigma)](g(\sigma)-g(0))
 \end{equation}
and
 \begin{equation}\label{V4}
   \int_{t}^\infty d_{\sigma}[\mathscr{V}(t)(Id-P(t))\mathscr{V}^{-1}(\sigma)](g(\sigma)-g(0))
 \end{equation}
exist for all $t\in\mathbb{R}$ and the maps \eqref{V3} and \eqref{V4} are bounded, then
\begin{equation*}
  \frac{dz}{d\tau}=D[\Lambda(t)z+g(t)]
\end{equation*}
 has a unique bounded solution.
\end{lemma}

\begin{remark}
According to \cite[Remark 4.11]{BFS-JDE2018},
if $g$ is bounded with $\|g(t)\|\leq M$
and the condition (A2) holds, then Bonotto et al. \cite{BFS-JDE2018} proved the existence of integral equations \eqref{V3}
and \eqref{V4}. In addition,  they obtained that
\[
\sup_{t\in\mathbb{R}}\left\|  \int_{-\infty}^t d_{\sigma}[\mathscr{V}(t)P(t)\mathscr{V}^{-1}(\sigma)](g(\sigma)-g(0))\right\| \leq 2MK^2C^3 e^{3CV_\Lambda}V_\Lambda^2
\]
and
\[
\sup_{t\in\mathbb{R}}\left\| \int_{t}^\infty d_{\sigma}[\mathscr{V}(t)(Id-P(t))\mathscr{V}^{-1}(\sigma)](g(\sigma)-g(0))\right\| \leq 2MK(1+K)C^3 e^{3CV_\Lambda}V_\Lambda^2.
\]
\end{remark}

\section{Stable manifold for the generalized ODEs}

\subsection{Main result}
In this paper, we investigate the nonlinear generalized ODEs given by
\begin{equation}\label{NL}
\frac{dz}{d\tau}=D[\Lambda(t)z+F(z,t)],
\end{equation}
where $F: \mathscr{Z}\times \mathbb{R}\to \mathscr{Z}$ is Kurzweil integrable satisfying $F(0,t)=0$.
In the following, we make the  assumptions on $\Lambda$ and $F$:
\\
(H1) there exists a positive constant $C_a>0$ such that
$$
\|[Id+(\Lambda(t^+)-\Lambda(t))]^{-1}\|\leq C_a, \quad  \|[Id+(\Lambda(t)-\Lambda(t^-))]^{-1}\|\leq C_a,
$$
and
$$
V_\Lambda:=\sup{\mathrm{var}_c^d \Lambda: c,d\in\mathbb{R}, c<d}<\infty;
$$
\\
(H2) the function $F\in\mathscr{F}(\Omega,h)$, where $h:\mathbb{R}\to \mathbb{R}$ is a nondecreasing function such that
$$
V_h:=\sup\{\mathrm{var}_c^d h: c,d\in\mathbb{R}, c<d\}<\infty.
$$

%Remark that   (H2) here  represents a Lipschitz-type condition.
Our goal is to establish the existence of a stable invariant manifold for the nonlinear generalized ODEs \eqref{NL}.
We first present the definition of the stable manifold of generalized ODEs .
Let $\mathscr{E}^s(s)$ and $\mathscr{E}^u(s)$  be  the stable and unstable spaces at time $s$, respectively.
  Given a constant $C >0$, let $\mathscr{X}$ be the set of all maps
\[
m:\{(s,\zeta),s\in\mathbb{R},\zeta\in \mathscr{E}^s(s)\}\to \mathscr{Z}
\]
satisfying the  properties:
(i) for any $s\in\mathbb{R}$, $m(s,0)=0$;
(ii) for any $s\in\mathbb{R}$ and $\zeta\in \mathscr{E}^s(s)$, $m(s,\zeta)\in \mathscr{E}^u (s)$;
(iii) for any $s\in\mathbb{R}$ and $\zeta,\widetilde{\zeta}\in \mathscr{E}^s(s)$,
\begin{equation}\label{LIP}
  \|m(s,\zeta)-m(s,\widetilde{\zeta})\|\leq C \, \|\zeta-\widetilde{\zeta}\|.
\end{equation}
Given  a $m\in\mathscr{X}$. Then the stable manifold  of \eqref{NL} is defined by
\[
 \mathscr{M}_{gra}:=\{(s,\zeta,m(s,\zeta)):s\in\mathbb{R},\zeta\in \mathscr{E}^s(s)\}.
\]

Now we are in a position to state our main result in this paper.

\begin{theorem}\label{theorem1}
Suppose that the linear generalized ODEs \eqref{LGODE} admit  an exponential dichotomy.
If conditions (H1) and (H2) hold, and the Lpschitz-type constant $V_h$ in (H2) is sufficiently small,
then Eq. \eqref{NL} has a stable manifold $\mathscr{M}_{gra}$. Furthermore,  $\mathscr{M}_{gra}$ has the following properties:
\begin{description}
  \item[(i)]  for all $t\in\mathbb{R}$,  any solution $z(t)$ of Eq. \eqref{NL} with  $(s,z(s))\in \mathscr{M}_{gra}$ satisfies
$(t,z(t))\in \mathscr{M}_{gra}$,
  \item[(ii)]   any solution $z(t)$ of Eq. \eqref{NL} with  $(s,z(s))\not\in \mathscr{M}_{gra}$ is unbounded on $[s,\infty)$.
\end{description}
\end{theorem}

%\[
%2V_h(1+K(1+2K))C_a^3e^{3C_aV_\Lambda}V_\Lambda^2<1,
%\]

Remark that in Theorem \ref{theorem1}, we establish the stable invariant manifold for sufficiently small perturbations of an exponential dichotomy, where property
(i) implies that the stable manifold of Eq. \eqref{NL} has invariance, and property  (ii) states that
 any solutions of Eq. \eqref{NL}  outside the stable manifold are unbounded.
We  will employ the Lyapunov-Perron method to obtain Theorem \ref{theorem1}, where the greatest difficulty is to establish the Lyapunov-Perron equation for
Eq. \eqref{NL} under the Kurzweil integral theory.
Therefore, we will overcome this in Lemma \ref{lemma1} and point out the differences between the Lyapunov-Perron equation under the Kurzweil integral
 and the classical Riemann integral in Remark \ref{rem111}.
On the other hand,  if the Kurzweil integral reduces to the Riemann integral, then our result is the classical result of stable manifolds
under ODEs  (see Hale \cite{Hale-book-1969}).
  Furthermore,  in the next section, we will also apply this result to the MDEs and the IDEs in the sense of Lebesgue (or Lebesgue-Stieltjes) integral.

\subsection{Auxiliary lemmas}

In order to prove Theorem \ref{theorem1}, in this section we present some auxiliary   lemmas.
For convenience, in what follows we write $\mathscr{V}(t,s)=\mathscr{V}(t)\mathscr{V}^{-1}(s)$ for all $t,s\in\mathbb{R}$.
We first give a generalized Lyapunov-Perron equation of Eq. \eqref{NL}, i.e., we have the following result.

\begin{lemma}\label{lemma1}
If  $z:[s,\infty)\to\mathscr{Z}$ is a bounded solution of Eq. \eqref{NL}, then for each $t\geq s$, we have
\begin{equation}\label{LPE}
\begin{split}
  z(t)=& \mathscr{V}(t,s)P(s)z(s)+\int_s^t DF(z(\tau),\gamma)-\int_{s}^t d_{\sigma}[\mathscr{V}(t,\sigma)P(\sigma)]\int_s^\sigma DF(z(\tau),\gamma) \\
  &+ \int_{t}^\infty d_{\sigma}[\mathscr{V}(t,\sigma)(Id-P(\sigma))] \int_s^\sigma DF(z(\tau),\gamma).
\end{split}
\end{equation}
Conversely, all  solutions of Eq. \eqref{LPE} on $[s,\infty)$ are solutions of Eq. \eqref{NL}.
\end{lemma}

\begin{proof}
We firstly claim that integral \eqref{LPE} holds.
   Since $z(t)$ is a solution of \eqref{NL}, we derive from Lemma \ref{changshu}  that
\[
z(t)=\mathscr{V}(t,s)z(s)+\int_s^t DF(z(\tau),\gamma)- \int_{s}^t d_{\sigma}[\mathscr{V}(t,\sigma)]\int_s^\sigma DF(z(\tau),\gamma),
\]
for all $t\geq s$. We  split $z(t)$ into two parts: $P(t)z(t)$ and $(Id-P(t))z(t)$. Then
\begin{equation}\label{LPE1}
P(t)z(t)=P(t)\mathscr{V}(t,s)z(s)+\int_s^t P(t)DF(z(\tau),\gamma)- \int_{s}^t d_{\sigma}[P(t)\mathscr{V}(t,\sigma)]\int_s^\sigma DF(z(\tau),\gamma)
\end{equation}
and
\[\begin{split}
(Id-P(t))z(t)=& (Id-P(t))\mathscr{V}(t,s)z(s)+\int_s^t (Id-P(t))DF(z(\tau),\gamma)\\
&- \int_{s}^t d_{\sigma}[(Id-P(t))\mathscr{V}(t,\sigma)]\int_s^\sigma DF(z(\tau),\gamma).
\end{split}\]
It is obvious that
\begin{equation}\label{LP0}
\begin{split}
 (Id-P(s))z(s)=& \mathscr{V}(s,t)(Id-P(t))z(t)-\int_s^t \mathscr{V}(s,t)(Id-P(t))DF(z(\tau),\gamma) \\
 &+\int_s^t d_\sigma [\mathscr{V}(s,\sigma)(Id-P(\sigma))]\int_s^\sigma DF(z(\tau),\gamma).
\end{split}
\end{equation}
Since $z(t)$ is bounded on $[s,\infty)$, there is a positive constant $N>0$ satisfying $\|z(t)\|\leq N$.
%There exists a positive constant $N>0$ such that $\|z(t)\|\leq N$, due to $z(t)$ is bounded on $[s,\infty)$.
From \eqref{ED}, it follows that
\begin{equation}\label{LP1}
  \| \mathscr{V}(s,t)(Id-P(t))z(t)\|\leq Ke^{-\alpha(t-s)}\|z(t)\|\leq NKe^{-\alpha(t-s)}.
\end{equation}
By the definition of Kurzweil  integral, for each $\delta>0$, there exists a gauge $\delta$ of $[s,t]$ such that for every $\delta$-fine tagged division $D=\{(\tau_j,[t_{j-1},t_j]), j=1,2,\cdots,|D| \}$ of $[s,t]$, we have
\[
\left\|\int_s^t DF(z(\tau),\gamma)\right\|=\left\|\sum_{j=1}^{|D|} \left[F(z(\tau_j),t_j)-F(z(\tau_j),t_{j-1})\right]\right\|,
\]
and by using condition (H2), we have
\begin{equation*}
 \left\|\sum_{j=1}^{|D|} \left[F(z(\tau_j),t_j)-F(z(\tau_j),t_{j-1})\right]\right\| \leq \left| \sum_{j=1}^{|D|} h(t_j)-h(t_{j-1})\right|\leq |h(t)-h(s)|\leq 2V_h.
\end{equation*}
Thus,
\begin{equation}\label{LP2}
  \left\| \int_s^t \mathscr{V}(s,t)(Id-P(t))DF(z(\tau),\gamma)\right\|\leq Ke^{-\alpha(t-s)}\left\|\int_s^tDF(z(\tau),\gamma)\right\|
  \leq 2V_h Ke^{-\alpha(t-s)}
\end{equation}
and
\begin{equation}\label{LP3}
  \left\|\int_s^t d_\sigma [\mathscr{V}(s,\sigma)(Id-P(\sigma))]\int_s^\sigma DF(z(\tau),\gamma)\right\|\leq
\left\|\int_s^t d_\sigma [(Id-P(s))\mathscr{V}(s,\sigma)]\right\|\cdot 2V_h.
\end{equation}
We deal with the above Perron-Stieltjes integral. Indeed,
\[
\left\|\int_s^t d_\sigma [(Id-P(s))\mathscr{V}(s,\sigma)]\right\|=\left\|\int_s^t d_\sigma [\mathscr{V}(s)(Id-P(s))\mathscr{V}^{-1}(\sigma)]\right\|.
\]
By using  \cite[Remark 4.11]{BFS-JDE2018}, we deduce that
\[
2\sup\limits_{t\geq s}\left\|\int_s^t d_\sigma [\mathscr{V}(s)(Id-P(s))\mathscr{V}^{-1}(\sigma)]\right\|V_h \leq 2V_hK(1+K)C_a^3e^{3C_aV_\Lambda}V_\Lambda^2<\infty.
\]
Set $t\to\infty$, we obtain form \eqref{LP0}, \eqref{LP1}, \eqref{LP2} and \eqref{LP3} that
\[
(Id-P(s))z(s)=\int_s^\infty d_\sigma [\mathscr{V}(s,\sigma)(Id-P(\sigma))]\int_s^\sigma DF(z(\tau),\gamma)
\]
and
\[
\mathscr{V}(t,s)(Id-P(s))z(s)=\int_s^\infty d_\sigma [\mathscr{V}(t,\sigma)(Id-P(\sigma))]\int_s^\sigma DF(z(\tau),\gamma).
\]
Therefore, we get
\[
(Id-P(t))z(t)=\int_s^t (Id-P(t))DF(z(\tau),\gamma)+\int_t^\infty d_\sigma [\mathscr{V}(t,\sigma)(Id-P(\sigma))]\int_s^\sigma DF(z(\tau),\gamma).
\]
Combining this equality with \eqref{LPE1}, we conclude that
\begin{equation*}
\begin{split}
  z(t)=& \mathscr{V}(t,s)P(s)z(s)+\int_s^t DF(z(\tau),\gamma)-\int_{s}^t d_{\sigma}[\mathscr{V}(t,\sigma)P(\sigma)]\int_s^\sigma DF(z(\tau),\gamma) \\
  &+ \int_{t}^\infty d_{\sigma}[\mathscr{V}(t,\sigma)(Id-P(\sigma))] \int_s^\sigma DF(z(\tau),\gamma).
\end{split}
\end{equation*}

Conversely,   assume that \eqref{LPE} holds for any $t\geq s$. Then, we can deduce that
\begin{small}
\[\begin{split}
\mathscr{V}&(t,s) z(s)+\int_s^t DF(z(\tau),\gamma)-\int_{s}^t d_{\sigma}[\mathscr{V}(t,\sigma)]\int_s^\sigma DF(z(\tau),\gamma)\\
=& P(t)\mathscr{V}(t,s)z(s)+\int_s^t P(t) DF(z(\tau),\gamma)-\int_{s}^t d_{\sigma}[P(t)\mathscr{V}(t,\sigma)]\int_s^\sigma DF(z(\tau),\gamma) \\
&+(Id-P(t)) \mathscr{V}(t,s)z(s)+\int_s^t (Id-P(t)) DF(z(\tau),\gamma)
-\int_{s}^t d_{\sigma}[(Id-P(t))\mathscr{V}(t,\sigma)]\int_s^\sigma DF(z(\tau),\gamma) \\
=&  \mathscr{V}(t,s)P(s)z(s)+\int_s^t  DF(z(\tau),\gamma)-\int_{s}^t d_{\sigma}[\mathscr{V}(t,\sigma)P(\sigma)]\int_s^\sigma DF(z(\tau),\gamma) \\
&+\int_s^\infty d_{\sigma}[(Id-P(t))\mathscr{V}(t,\sigma)]\int_s^\sigma DF(z(\tau),\gamma)-\int_{s}^t d_{\sigma}[(Id-P(t))\mathscr{V}(t,\sigma)]\int_s^\sigma DF(z(\tau),\gamma) \\
=& \mathscr{V}(t,s)P(s)z(s)+\int_s^t  DF(z(\tau),\gamma)-\int_{s}^t d_{\sigma}[\mathscr{V}(t,\sigma)P(\sigma)]\int_s^\sigma DF(z(\tau),\gamma) \\
&+\int_{t}^\infty d_{\sigma}[\mathscr{V}(t,\sigma)(Id-P(\sigma))] \int_s^\sigma DF(z(\tau),\gamma)\\
=& z(t),
\end{split}\]
\end{small}
for $t\geq s$. This proves that $z(t)$ is a solution of Eq. \eqref{NL}, and the proof is completed.
\end{proof}

\begin{remark}\label{rem111}
Under the classical ODEs, the above so-called generalized Lyapunov-Perron integral (see \eqref{LPE}) is simplified as
\[
  z(t)= \mathscr{V}(t,s)P(s)z(s)+\int_{s}^t \mathscr{V}(t,\sigma)P(\sigma)F(z(\sigma ),\sigma ) d\sigma
  -\int_{t}^\infty \mathscr{V}(t,\sigma)(Id-P(\sigma))F(z(\sigma ),\sigma ) d\sigma,
\]
where $F(0,t)=0$ and $F(z,t)$ is Lipschitzian in $z$ with a sufficiently small constant $L_F$.
We see that the term $\int_s^t DF(z(\tau),\gamma)$ disappears, since the Kurzweil integral reduces to the Riemann integral.
%, we conclude that
%$\int_s^t DF(z(\tau),\gamma)=F(z(t),t)$ and $\|F(z(t),t)-F(0,t)\|\leq L_F \|z(t)\|$.
\end{remark}

\begin{lemma}\label{lemma2}
Assume that Eq. \eqref{LGODE} admits an exponential dichotomy. Define
\begin{equation}\label{Hf}
 (\mathscr{H}f)(t)= -\int_s^t d_{\sigma}[\mathscr{V}(t,\sigma)P(\sigma)] (f(\sigma)-f(s))
 +\int_{t}^\infty d_{\sigma}[\mathscr{V}(t,\sigma)(Id-P(\sigma))](f(\sigma)-f(s)),
\end{equation}
where $f(t)$ is a bounded function with $\|f(t)\|\leq N$ for any $t\geq s$ and $N>0$.
If   (H1) and (H2) hold, then $\mathscr{H}$ is a bounded linear operator.% and the norm
%$\|\mathscr{H}\|\leq 2NK(1+2K)C_a^3 e^{3C_a V_\Lambda}V_\Lambda^2$.
\end{lemma}

\begin{proof}
Using the proof of Lemma \ref{lemma1}, the operator $\mathscr{H}$ is meaningful and linear. Combining (H1)-(H2) and   \cite[Remark 4.11]{BFS-JDE2018}, we have
\[
\sup\limits_{t\geq s} \left\|\int_s^t d_{\sigma}[\mathscr{V}(t,\sigma)P(\sigma)] (f(\sigma)-f(s))\right\|\leq 2NK^2 C_a^3 e^{3C_a V_\Lambda}V_\Lambda^2
\]
and
\[
\sup\limits_{t\geq s} \left\|\int_{t}^\infty d_{\sigma}[\mathscr{V}(t,\sigma)(Id-P(\sigma))](f(\sigma)-f(s))\right\|
\leq 2NK(1+K)C_a^3 e^{3C_aV_\Lambda}V_\Lambda^2.
\]
Hence, one can immediately obtain that
$\|\mathscr{H}\|\leq 2NK(1+2K)C_a^3 e^{3C_a V_\Lambda}V_\Lambda^2. $ This completes the proof.
\end{proof}

\begin{lemma}\label{lemma3}
Assume that Eq. \eqref{LGODE} admits an exponential dichotomy.  If  (H1)-(H2) hold and the Lipschitz-type constant $V_h$ in (H2) is sufficiently small,
%\[
%2V_h(1+K(1+2K))C_a^3e^{3C_aV_\Lambda}V_\Lambda^2<1,
%\]
then for any $s\in\mathbb{R}$ and $\zeta\in\mathscr{E}^s (s)$,
the integral equation \eqref{LPE} with   $P(s)z(s)=\zeta$ has a unique  bounded solution
$\phi_t (s,\zeta)$.
Moreover, $\phi(s,0)=0$ and $\phi(s,\zeta)$ is Lipschitzian in $\zeta$.
\end{lemma}

\begin{proof}
  Let $\Theta$ be the set  of all functions $z:[s,\infty)\to\mathscr{Z}$ satisfying $\|z(t)\|<\infty$ for all $t\geq s$.
  We  define  the operator $\mathscr{J}$ as follows:
 \[
 (\mathscr{J}_\zeta z)(t):=\mathscr{V}(t,s)P(s)\zeta+(\widehat{f}(t))_z+\mathscr{H}(\widehat{f}(t))_z,
 \]
where $\zeta=P(s)z(s)$, $(\widehat{f}(t))_z=\int_s^t DF(z(\tau),\gamma)$ and $\mathscr{H}$ is defined by Lemma \ref{lemma2}.
 Now we claim that $\mathscr{J}_\zeta$ is a contraction map on $\Theta$ uniformly with respect to $\zeta$.
  Firstly, $\mathscr{J}_\zeta$ maps set $\Theta$ into itself.
  In fact, we have
\[\begin{split}
 \sup\limits_{t\geq s} \|\mathscr{H}(\widehat{f}(t))_z\|\leq & \sup\limits_{t\geq s}
 \left\|\int_s^t d_{\sigma}[\mathscr{V}(t,\sigma)P(\sigma)] \int_s^\sigma DF(z(\tau),\gamma)\right\|\\
 &+ \sup\limits_{t\geq s} \left\|\int_{t}^\infty d_{\sigma}[\mathscr{V}(t,\sigma)(Id-P(\sigma))]\int_s^\sigma DF(z(\tau),\gamma)\right\|\\
 \leq& \sup\limits_{t\geq s} \left\|\int_s^t d_{\sigma}[\mathscr{V}(t,\sigma)P(\sigma)]\right\| | h(\sigma)-h(s)|\\
 &+ \sup\limits_{t\geq s} \left\|\int_{t}^\infty d_{\sigma}[\mathscr{V}(t,\sigma)(Id-P(\sigma))]\right\| |h(\sigma)-h(s)|\\
 \leq& 2V_h K(1+2K)C_a^3 e^{3C_a V_\Lambda}V_\Lambda^2,
\end{split}\]
 due to Lemma \ref{lemma2} and condition (H2). Then
 \[
  \sup\limits_{t\geq s}\| (\mathscr{J}_\zeta z)(t)\|\leq Ke^{-\alpha(t-s)}\|\zeta\|+2V_h+2V_h K(1+2K)C_a^3 e^{3C_a V_\Lambda}V_\Lambda^2<\infty.
 \]
 Taking arbitrary $z_1(t), z_2(t)\in \Theta$, from the definition of Kurzweil integral and   condition (H2), it follows that
\begin{equation*}
\begin{split}
  &\left\|\int_s^t DF(z_1(\tau),\gamma)-DF(z_2(\tau),\gamma)\right\|\\
  =&
\left\|\sum_{j=1}^{|D|}\left[F(z_1(\tau_j),t_j)-F(z_1(\tau_j),t_{j-1})\right]-\left[F(z_2(\tau_j),t_j)-F(z_2(\tau_j),t_{j-1})\right]\right\|  \\
\leq& \sum_{j=1}^{|D|}\|z_1(\tau_j)-z_2(\tau_j)\|\cdot |h(s_j)-h(s_{j-1})| \\
=& \int_s^t \|z_1(\tau)-z_2(\tau)\|dh(\tau)\\
\leq& \sup\limits_{\tau\geq s}\|z_1(\tau)-z_2(\tau)\| 2V_h.
\end{split}
\end{equation*}
 Then we have
\[\begin{split}
  \sup\limits_{t\geq s}\| (\mathscr{J}_\zeta z_1-\mathscr{J}_\zeta z_2)(t)\|\leq&  \sup\limits_{t\geq s}\left\|\int_s^t DF(z_1(\tau),\gamma)-DF(z_2(\tau),\gamma)\right\|\\
 &+\sup\limits_{t\geq s}\left\|\int_s^t d_{\sigma}[\mathscr{V}(t,\sigma)P(\sigma)] \int_s^\sigma (DF(z_1(\tau),\gamma)-DF(z_2(\tau),\gamma)) \right\|\\
 &+\sup\limits_{t\geq s} \left\|\int_{t}^\infty d_{\sigma}[\mathscr{V}(t,\sigma)(Id-P(\sigma))]\int_s^\sigma  (DF(z_1(\tau),\gamma)-DF(z_2(\tau),\gamma)) \right\|\\
 \leq& 2V_h (1+ K(1+2K))C_a^3 e^{3C_a V_\Lambda}V_\Lambda^2 \cdot \sup\limits_{t\geq s}\|z_1(t)-z_2(t)\|.
 \end{split}\]
Taking $V_h$ sufficiently small such that
$$
2V_h (1+ K(1+2K))C_a^3 e^{3C_a V_\Lambda}V_\Lambda^2<1.
$$
Then $\mathscr{J}_\zeta$ is contractive and it has a unique fixed point $\phi_t (s,\zeta)\in\Theta$.
   For any $\zeta_1,\zeta_2\in\mathscr{E}^s(s)$ and $z(s)\in\Theta$, we get
 \[
 \| (\mathscr{J}_{\zeta_1} z-\mathscr{J}_{\zeta_2} z)(s)\|\leq Ke^{-\alpha(s-s)}\|\zeta_1-\zeta_2\|=K\|\zeta_1-\zeta_2\|.
 \]
We write $\phi(s,\zeta)=\phi_{t}(s,\zeta)$. It is clear that
\[\begin{split}
\|\phi(s,\zeta_1)-\phi(s,\zeta_2)\|=& \|(\mathscr{J}_{\zeta_1} (\phi(s,\zeta_1))-\mathscr{J}_{\zeta_2} (\phi(s,\zeta_2)))(s)\| \\
\leq& \|(\mathscr{J}_{\zeta_1} (\phi(s,\zeta_1))-\mathscr{J}_{\zeta_1} (\phi(s,\zeta_2)))(s)\|+\|(\mathscr{J}_{\zeta_1} (\phi(s,\zeta_2))-\mathscr{J}_{\zeta_2} (\phi(s,\zeta_2)))(s)\| \\
\leq& 2V_h (1+ K(1+2K))C_a^3 e^{3C_a V_\Lambda}V_\Lambda^2 \cdot \|\phi(s,\zeta_1)-\phi(s,\zeta_2)\|+K\|\zeta_1-\zeta_2\|,
\end{split}\]
and thus,
 \[
 \|\phi(s,\zeta_1)-\phi(s,\zeta_2)\|\leq \frac{K}{1-\mathscr{L}} \|\zeta_1-\zeta_2\|,
 \]
 where
$$
\mathscr{L}:=2V_h (1+ K(1+2K))C_a^3 e^{3C_a V_\Lambda}V_\Lambda^2<1.
$$
 This implies that $\phi(s,\zeta)$ is Lipschitz continuous in $\zeta$.
   By the uniqueness of solution of Eq. \eqref{NL}, we have $\phi(s,0)=0$, and the proof is completed.
\end{proof}

\subsection{Proof of  Main Theorem}

\begin{proof}[Proof of Theorem \ref{theorem1}]
For every fixed $s\in\mathbb{R}$ and $\zeta\in \mathscr{E}^s(s)$, we obtain from Lemma \ref{lemma3} that there is a unique solution $\phi_t(s,\zeta)\in\Theta$ such that $\phi(s,0)=0$ and
\begin{equation}\label{FL}
\begin{split}
  \phi_t(s,\zeta)=& \mathscr{V}(t,s)P(s)\zeta+\int_s^t DF(z(\tau),\gamma)-\int_s^t d_{\sigma}[\mathscr{V}(t,\sigma)P(\sigma)] \int_s^\sigma DF(z(\tau),\gamma)\\
  &+\int_{t}^\infty d_{\sigma}[\mathscr{V}(t,\sigma)(Id-P(\sigma))]\int_s^\sigma DF(z(\tau),\gamma).
\end{split}
\end{equation}
Define
\[
 \mathscr{M}_{gra}:=\{(s,z(s))\in\mathbb{R}\times\mathscr{Z}:z(t,s,z(s))\; \mathrm{is} \; \mathrm{defined} \; \mathrm{in} \;\mathrm{ the} \;\mathrm{ set} \; \Theta\}.
\]
We derive from Lemma \ref{lemma1} and \eqref{FL} that the initial value $z(s)$, which  compose the set $ \mathscr{M}_{gra}$ can be expressed as:
\[
z(s)=\phi(s,\zeta)=\zeta+\int_s^\infty d_{\sigma}[\mathscr{V}(s,\sigma)(Id-P(\sigma))]\int_s^\sigma DF(z(\tau),\gamma):=\zeta+m(s,\zeta).
\]
Since $\phi(s,\zeta)$ is Lipschitz continuous with respect to $\zeta$, for any $\zeta_1,\zeta_2\in\mathscr{E}^s(s)$, we have
\[\begin{split}
\|m(s,\zeta_1)-m(s,\zeta_2)\|
\leq& \left\|\int_{s}^\infty d_{\sigma}[\mathscr{V}(s,\sigma)(Id-P(\sigma))]\int_s^\sigma  (DF(\phi(\tau,\zeta_1),\gamma)-DF(\phi(\tau,\zeta_2),\gamma)) \right\|\\
\leq&  \left\|\int_{s}^\infty d_{\sigma}[\mathscr{V}(s,\sigma)(Id-P(\sigma))]\int_s^\sigma \|\phi(\tau,\zeta_1)-\phi(\tau,\zeta_2)\|dh(\gamma) \right\|\\
\leq& 2V_h (1+K)KC_a^3 e^{3C_a V_\Lambda}V_\Lambda^2\cdot \frac{K\|\zeta_1-\zeta_2\| }{1-2V_h (1+ K(1+2K))C_a^3 e^{3C_a V_\Lambda}V_\Lambda^2}\\
\leq& \widetilde{C}\|\zeta_1-\zeta_2\|,
\end{split}\]
for some constant $\widetilde{C}>0$.
  Therefore, $m(s,\zeta)$ is Lipschitz continuous in $\zeta$.
  In addition, since $\phi(s,0)=0$ and $F(0,s)=0$, we have $m(s,0)=0$.
Therefore, $\mathscr{M}_{gra}$ is a stable manifold.

 For the invariance of the stable manifold, we proceed as follows.
  Given $(s,z(s))\in \mathscr{M}_{gra}$, Eq. \eqref{NL} has a unique solution $z(t)=z(t,s,z(s))$.
  Obviously, $z(t,s,z(s))\in\Theta$ and $(t,z(t))\in \mathscr{M}_{gra}$ for any $t\geq s$.
  We  claim that $(t,z(t))\in \mathscr{M}_{gra}$ for any $t<s$.
  In fact, let $t_1<s$ and $w=z(t_1,s,z(s))$.
  Then Eq. \eqref{NL} has a unique solution $z(t,t_1,w)$ with the initial value $(t_1,w)$.
  It is clear that $(s,z(s,t_1,w))=(s,z(s))\in \mathscr{M}_{gra}$ and $\sup_{t\in[t_1,s]}\|z(t,t_1,w)\|\leq \widetilde{N}$ for some $\widetilde{N}>0$.
  Therefore, $z(t,t_1,w)\in\Theta$.
  In view of definition of $\mathscr{M}_{gra}$, we see that
$(t_1,w)=(t_1,z(t_1,s,z(s)))\in \mathscr{M}_{gra}$.

    Next, we prove the unboundedness of the solutions outside of the stable manifolds. %that any solutions  outside the stable manifold are unbounded.
     That is, we prove that for some $s\in\mathbb{R}$, any solutions $z(t)$ of Eq. \eqref{NL} with $(s,z(s))\not\in \mathscr{M}_{gra}$
is unbounded on $[s,\infty)$.
   By way of contradiction, in fact, from Lemma \ref{lemma1}, if we suppose that $z(t)$ is a bounded solution of Eq. \eqref{LPE} for any $t\geq s$.
   Then, $z(t)=P(t)z(t)+(Id-P(t))z(t)=u(t)+v(t)$, i.e.,
 \[
 u(t)=P(t)\mathscr{V}(t,s)u(s)+\int_s^t P(t)DF(z(\tau),\gamma)-\int_s^t d_\sigma [\mathscr{V}(t,\sigma)P(\sigma)]\int_s^\sigma DF(z(\tau),\gamma)
 \]
 and
 \begin{equation}\label{vvv}
   v(t)=(Id-P(t))\mathscr{V}(t,s)x+\int_s^t (Id-P(t))DF(z(\tau),\gamma)+\int_t^\infty  [\mathscr{V}(t,\sigma)(Id-P(\sigma))]\int_s^\sigma DF(z(\tau),\gamma),
 \end{equation}
where
\[
x=v(s)-\int_s^\infty d_\sigma [\mathscr{V}(t,\sigma)(Id-P(\sigma))]\int_s^\sigma DF(z(\tau),\gamma).
\]
 For any $t\geq s$, since the solution is bounded, \eqref{vvv} converges and is uniformly bounded.
  However,
\[
\|(Id-P(t))\mathscr{V}(t,s)x\|\geq K^{-1} e^{\alpha(t-s)}\|x\|, \quad \mathrm{for} \; t\geq s,
\]
which implies that $x=0$ since $z(t)$ is bounded. Therefore, \eqref{LPE} is valid.
However, the bounded solution of this system is unique, and satisfies $(s,z(s))\in\mathscr{M}_{gra}$. This contradicts to our assumption. Therefore, outside of the stable manifolds, any solutions are unbounded.
Consequently, the proof  is completed.
\end{proof}

\section{Applications}

\subsection{Stable manifold theorem for measure differential equations (MDEs)}
\subsubsection{Basic concepts for MDEs}
Denote that $\mathscr{Z}$  is a Banach space and  $I\subset \mathbb{R}$ is an interval.
 % Let $I\subset \mathbb{R}$ be an interval and $\mathscr{Z}$ be a Banach space.
 Consider the   linear  MDEs
\begin{equation}\label{ML}
{   Dz=\mathscr{A}(t)z+\mathscr{C}(t)Du,}
\end{equation}
%{\color{red}(If you use Perron integral, ---see assumptions below--- you cannot use distributional derivatives. I will find a reference where this problem is discussed. Here one must use the integral form of this equation. The same observation holds for equation~\eqref{MNL}, on page~\pageref{MNL})}
%{\color{blue}
%For these basic concepts, we refer to references [Bonotto et al. \cite{BFS-JDDE2020}, Section 4.1-Measure differential equations] and [Bonotto et al. \cite{BFS-JDE2018}, Section 5-Applications]. The same is true for some of the different integrals involved below.
%}
where $Dz$ and $Du$ denote the distributional derivatives of $z$ and $u$ in the sense of Schwartz, and the functions
$\mathscr{A}: I\to \mathscr{B}(\mathscr{Z})$, $\mathscr{C}: I\to \mathscr{B}(\mathscr{Z})$ and $u:I\to\mathbb{R}$
satisfy:
\\
(D1) $\mathscr{A}(t)$ is Perron integrable for any  $t\in I$.\\
(D2) $u(t)$ is of locally bounded variation for any  $t\in I$ and continuous from the left on $I\backslash \{\inf I\}$. \\
(D3) $du$ denotes the Lebesgue-Stieltjes measure generated by the
function $u$, $\mathscr{C}(t)$ is Perron-Stieltjes integrable in  $u$ for any  $t\in I$.\\
Moreover, we consider some additional assumptions:\\
(D4) There exists a {Lebesgue} measurable function
%{(\color{red}why Lebesgue, and not Perron?)}
$m_1:I\to\mathbb{R}$ satisfying for any $c,d \in I$, we have
$\int_c^d m_1(s)ds<\infty$ and
\[
\left\|\int_c^d \mathscr{A}(s)ds\right\|\leq \int_c^d m_1(s)ds.
\]
(D5) There exists a function $du$-measurable $m_2:I\to\mathbb{R}$ satisfying for any $c,d \in I$,  we have
$\int_c^d m_2(s)du(s)<\infty$ and
\[
\left\|\int_c^d \mathscr{C}(s)du(s)\right\|\leq \int_c^d m_2(s)du(s).
\]
(D6) For all $t$ such that $t$ is a point of discontinuity of $u$, we have
\[
\left(Id+\lim\limits_{r\to t^+}\int_t^r\mathscr{C}(s)du(s)\right)^{-1}\in \mathscr{B}(\mathscr{Z}).
\]

By the assumptions (D1)--(D3), for the initial value $z(t_0)=z_0$, we say that $z:[c,d]\subset I\to \mathscr{Z}$ is a solution of \eqref{ML}, if
\[
z(t)=z_0+\int_{t_0}^{t} \mathscr{A}(s)z(s)ds+\int_{t_0}^t \mathscr{C}(s)z(s)du(s).
\]
   If all conditions (D1)--(D6) hold, then the existence and uniqueness of a solution of \eqref{ML} associated to the initial value $z(t_0)=z_0$ follows from Theorem 5.2 in \cite{BFS-JDE2018} immediately. Hence, conditions (D1)--(D6) are always satisfied throughout this subsection.

\begin{lemma}(\cite{S-BOOK1992} Theorem 5.17)
 Given $t_0\in [c,d]$, for an initial condition $z(t_0)=z_0$, the function $z:[c,d]\subset I\to \mathscr{Z}$ is a solution of \eqref{ML},  iff  $z$
is a solution of
\begin{equation}\label{GGL}
  \begin{cases}
  \frac{dz}{d\tau}=D[\Lambda(t)z+G(t)z],\\
  z(t_0)=z_0,
  \end{cases}
\end{equation}
where $\Lambda(t)=\int_{t_0}^{t} \mathscr{A}(s)ds$ and $G(t)=\int_{t_0}^t \mathscr{C}(s)du(s)$.
\end{lemma}

    The fundamental operator $U:I\times I\to \mathscr{B}(\mathscr{Z})$ of MDEs \eqref{ML} was given by \cite{BFS-JDE2018} satifying
 such that
\begin{equation}\label{MMU}
  U(t,s)=Id+\int_s^t \mathscr{A}(r)U(r,s)dr+\int_s^t \mathscr{C}(r)U(r,s)du(r), \quad t,s\in I.
\end{equation}
Additionally,  given $s\in I$, $U(\cdot,s)$ is an operator of locally bounded variation in $I$, and $z(t)=U(t,t_0)z_0$ is the solution of
\eqref{ML} satisfying   $z(t_0)=z_0\in \mathscr{Z}$.

\begin{definition}\label{MED}\cite{BFS-JDE2018}
The MDEs \eqref{ML} admit an exponential dichotomy with $(P,K,\alpha)$ on $I$,  if there exist  a projection $P:I \to \mathscr{Z}$ and
constants $K, \alpha$  satisfying
\[\begin{cases}
 \|\mathscr{U}(t) P(t)\mathscr{U}^{-1}(s)\|\leq K e^{-\alpha(t-s)}, \quad t\geq s;\\
 \|\mathscr{U}(t)(Id-P(t))\mathscr{U}^{-1}(s)\|\leq Ke^{\alpha(t-s)}, \quad t< s,
\end{cases}\]
where $\mathscr{U}(t)=U(t,t_0)$ and $\mathscr{U}^{-1}(t)=U(t_0,t)$.
\end{definition}

\begin{lemma}\label{lemma-ed}
(\cite{BFS-JDE2018} Proposition 5.7)
The MDEs \eqref{ML} admit an exponential dichotomy with $(P,K,\alpha)$ iff the generalized ODEs
\begin{equation}\label{GGODE}
  \frac{dz}{d\tau}=D[\Lambda(t)z+G(t)z], \quad t\in I,
\end{equation}
admit an exponential dichotomy,
where $\Lambda(t)=\int_{t_0}^{t} \mathscr{A}(s)ds$ and $G(t)=\int_{t_0}^t \mathscr{C}(s)du(s)$.
\end{lemma}

\subsubsection{Main result}
Consider the nonlinear MDE as follows.
\begin{equation}\label{MNL}
{ Dz=\mathscr{A}(t)z+\mathscr{C}(t)zDu+\mathscr{H}(t,z)Du,}
\end{equation}
where $\mathscr{H}(t,z):\mathbb{R} \times \mathscr{Z} \to \mathscr{Z}$ is  {Lebesgue-Stieltjes} integrable with respect to $u$ and $\mathscr{H}(t,0)=0$.
We suppose that the following conditions hold:\\
(a)  for all $t$ such that $t$ is a point of discontinuity of $u$, there exists a positive constant $C_g>0$ satisfying
\[
\left\|\left(Id+\lim\limits_{r\to t^+}\int_t^r\mathscr{C}(s)du(s)\right)^{-1}\right\|\leq C_g;
\]
(b)  $u$ is a bounded variation function on $\mathbb{R}$ and { $u$ is nondecreasing,} i.e.,
\[
V_u:=\sup\{\mathrm{var}_c^d u: c,d\in\mathbb{R}, c<d\}<\infty;
\]
(c) for any $t\in\mathbb{R}$ and $z\in\mathscr{Z}$,
 the function $\mathscr{H}(t,z)$ is uniformly bounded with some $M_H>0$ and is Lipschitz continuous in $z$ with a sufficiently small Lipschitz constant $L_H$.

We are now ready to establish the stable invariant manifold theorem for the nonlinear MDE \eqref{MNL}.
\begin{theorem}
Suppose that the MDEs \eqref{ML} possess an exponential dichotomy with $(P,K,\alpha)$ on $\mathbb{R}$. If conditions (a), (b), (c) hold and
 the  Lipschitz constant $L_H$ in (c) is sufficiently small,
%\[
%2 L_H V_u (1+K(1+2K))C_g^3e^{3C_g V_{\Lambda+G}}V_{\Lambda+G}^2<1,
%\]
then Eq. \eqref{MNL} has a stable invariant manifold $\mathscr{M}_{gra}$.
\end{theorem}

\begin{proof}
Let $t,t_0\in\mathbb{R}$ and we define the Kurzweil integrable map $N:\mathscr{Z}\times \mathbb{R}\to  \mathscr{Z} $ by
\begin{equation}\label{NH}
  N(z(t),t):=\int_{t_0}^t \mathscr{H}(s,z(s))du(s).
\end{equation}
The solution of Eq. \eqref{MNL} with the initial value $z(t_0)=z_0$ is given by
\[
z(t)=z_0+\int_{t_0}^{t} \mathscr{A}(s)z(s) ds+\int_{t_0}^t \mathscr{C}(s)z(s)du(s)+\int_{t_0}^t \mathscr{H}(s,z(s))du(s).
\]
Set $\Lambda(t)=\int_{t_0}^{t} \mathscr{A}(s)ds$, $G(t)=\int_{t_0}^t \mathscr{C}(s)du(s)$, and by \eqref{NH}, we obtain that
\[
z(t)=z_0+\int_{t_0}^t d[\Lambda(s)]z(s)+\int_{t_0}^t d[{G}(s)]z(s)+\int_{t_0}^t DN(z(s),s),
\]
which is a solution of the nonlinear generalized ODEs with the initial value $z(t_0)=z_0$,
\begin{equation}\label{NGODE}
\frac{dz}{d\tau}=D[\Lambda(t)z+G(t)z+N(z,t)].
\end{equation}
Since the MDEs \eqref{ML} possess an exponential dichotomy with $(P,K,\alpha)$, we derive from Lemma \ref{lemma-ed} that
\[
\frac{dz}{d\tau}=D[\Lambda(t)z+G(t)z]
\]
%{\color{red}(I think all these $F$ should be $G$)}
has an exponential dichotomy with the same  $(P,K,\alpha)$ on $\mathbb{R}$.

  Now we claim that $N(z,t)$ belongs to the class $\mathscr{F}(\Omega, u)$,
  here $\Omega=\mathscr{Z}\times \mathbb{R}$.
In fact, for any $t,\widetilde{t}\in\mathbb{R}$ and $z\in\mathscr{Z}$, by condition (c) (only using condition that $\mathscr{H}$ is bounded), we have
\[
\|N(z,t)-N(z,\widetilde{t})\|= \left\| \int_{\widetilde{t}}^t \mathscr{H}(s,z)du(s)\right\|\leq \|\mathscr{H}({\color{blue}\tau},z)\| |u(t)-u(\widetilde{t})|\leq M_H |u(t)-u(\widetilde{t})|,
\]
%{\color{red}(bad notation, $s$ is the integration variable. Same problem below.)}
and for any $t,\widetilde{t}\in\mathbb{R}$ and $z, \widetilde{z} \in\mathscr{Z}$, by condition (c) (using condition that $\mathscr{H}$ is bounded and Lipschitzian), we have
\[\begin{split}
 \|N(z,t)-N(z,\widetilde{t})-N(\widetilde{z},t)+N(\widetilde{z},\widetilde{t})\|= &
 \left\| \int_{\widetilde{t}}^t \left(\mathscr{H}(s,z)-\mathscr{H}(s,\widetilde{z})\right)du(s)\right\| \\
 \leq& \int_{\widetilde{t}}^t  \|\mathscr{H}(s,z)-\mathscr{H}(s,\widetilde{z})\|du(s) \\
 \leq& {  \|\mathscr{H}(\tau,z)-\mathscr{H}(\tau,\widetilde{z})\| |u(t)-u(\widetilde{t})| }\\
 \leq& L_H \|z-\widetilde{z}\| |u(t)-u(\widetilde{t})|.
\end{split}\]

    From Theorem 5.2 of \cite{BFS-JDE2018}, it follows that $\mathrm{var}_c^d (\Lambda+G)<\infty$ for all $c,d\in\mathbb{R}$ and $c<d$.
    For the sake of convenience, we write $V_{\Lambda+G}=\mathrm{var}_c^d (\Lambda+G)$.
    Indeed, let $D=\{t_0,t_1,\cdots,t_{|D|}\}$ be a division of $[c,d]$. Then
\[
\sum_{j=1}^{|D|} \|\Lambda(t_j)+G(t_j)-\Lambda(t_{j-1})-G(t_{j-1})\|\leq
\sum_{j=1}^{|D|} \left\| \int_{t_{j-1}}^{t_j} \mathscr{A}(s)ds \right\|
+\sum_{j=1}^{|D|} \left\| \int_{t_{j-1}}^{t_j} \mathscr{C}(s)du(s) \right\|,
\]
by using condition (D4) and (D5), we deduce that
\[
\sum_{j=1}^{|D|} \left\| \int_{t_{j-1}}^{t_j} \mathscr{A}(s)ds \right\|
+\sum_{j=1}^{|D|} \left\| \int_{t_{j-1}}^{t_j} \mathscr{C}(s)du(s) \right\|
\leq \int_c^d m_1(s)ds+\int_c^d m_2(s)du(s) <\infty,
\]
that is, $V_{\Lambda+G}<\infty$.
Hence, taking $L_H$ is sufficiently small, and by using conditions (a), (b) and (c), we can ensure that
$2 L_H V_u (1+K(1+2K))C_g^3e^{3C_g V_{\Lambda+G}}V_{\Lambda+G}^2<1.$
Consequently, all assumptions of Theorem \ref{theorem1} are valid, one concludes that Eq. \eqref{MNL} has a stable invariant manifold $\mathscr{M}_{gra}$.
This completes the proof.
\end{proof}

\subsection{Stable manifold theorem for IDEs}
\subsubsection{Basic concepts for IDEs}
Denote that $\mathscr{Z}$  is a Banach space and  $I\subset \mathbb{R}$ is an interval.
 Consider the  linear impulsive differential equation (for short, IDE)
\begin{equation}\label{IL}
\begin{cases}
\dot{z}(t)=\widetilde{A}(t)z(t),\quad t\neq t_i, \\
\vartriangle z(t_i)=z(t_i^+)-z(t_i)=B_i z(t_i),\quad i\in\mathscr{I}:=\{i\in\mathbb{Z}:t_i\in I\},
\end{cases}
\end{equation}
where $\widetilde{A}:I\to\mathscr{B}(\mathscr{Z})$ and $B_i\in \mathscr{B}(\mathscr{Z})$ satisfy the following assumptions:\\
(B1) $\widetilde{A}(t)$ is Perron integrable for any $t\in I$;\\
(B2) there is a { Lebesgue} measurable   function $m:I\to\mathbb{R}$ satisfyinf for any $c,d\in I$ and $c<d$, % the Lebesgue integral
$\int_c^d m(s)ds$ is finite and
\[
\left\|\int_c^d \widetilde{A}(s)ds \right\|\leq     \int_c^d m(s)ds.
\]
(B3) $(Id+B_i)^{-1}\in \mathscr{B}(\mathscr{Z})$, where $i\in\mathscr{I}$.\\
In addition, let $\{\cdots,t_{-k},\cdots, t_{-1},t_{1}, \cdots ,t_k,\cdots\}$ be the impulsive points satisfying the relation
\[
\cdots<t_{-k}<\cdots <t_{-1}<0<t_1<\cdots<t_k<\cdots,
\]
and $\lim\limits_{k\to \pm\infty}t_k=\pm\infty$. Set $\mathscr{I}_c^d:=\{i\in\mathscr{I}:c\leq t_i\leq d\}$ for $c,d\in I$.
Define   the Heaviside function $H_l$ by
\[
H_l(t)=
\begin{cases}
0 \quad \mathrm{if} \; t\leq l,\\
1 \quad \mathrm{if} \; t>l.
\end{cases}
\]
  Then, the solution of Eq. \eqref{IL} with the initial value $z(t_0)=z_0$ satisfying the following integral equation
\[
z(t)=\begin{cases}
 z_0+\int_{t_0}^t \widetilde{A}(s)z(s)ds+\sum\limits_{i\in\mathscr{I}_{t_0}^{t}}B_iz(t_i)H_{t_i}(t),\quad t\geq t_0 (t\in I),\\
 z_0+\int_{t_0}^t \widetilde{A}(s)z(s)ds-\sum\limits_{i\in\mathscr{I}_{t}^{t_0}}B_iz(t_i)(1-H_{t_i}(t)),\quad t< t_0 (t\in I).
\end{cases}
\]

{
\begin{lemma} (\cite[Theorem~4.8]{BFS-JDDE2020})
Let $t_0\in I$.  $x:I\to \mathscr{Z}$ is a solution of Eq. \eqref{IL} iff $z$ is a solution of the linear generalize ODE $\frac{dz}{d\tau}=D[\Lambda(t)z]$, where $\Lambda$ is given by
\begin{equation}\label{OA}
\Lambda(t)=\begin{cases}
 \int_{t_0}^t \widetilde{A}(s)ds+\sum\limits_{i\in\mathscr{I}_{t_0}^{t}}B_i H_{t_i}(t),\quad t\geq t_0,\\
  \int_{t_0}^t \widetilde{A}(s)ds-\sum\limits_{i\in\mathscr{I}_{t}^{t_0}}B_i(1-H_{t_i}(t)),\quad t< t_0.%\quad\text{This does not appear in \cite[Theorem~5.20]{S-BOOK1992}}}
\end{cases}
\end{equation}
\end{lemma}
}

Let $\Phi:I\times I\to \mathscr{B}(\mathscr{Z})$ denote the fundamental operator of the ODE $\dot{z}=\widetilde{A}(t)z$.
Define   $W:I\times I\to \mathscr{B}(\mathscr{Z})$ by
\[
W(t,s)=\Phi(t,t_k)\left( \prod_{k=i}^{j+1} [Id+B_k]\Phi(t_k,t_{k-1})\right)[Id+B_j]\Phi(t_j,s)
\]
if $t\geq s$, $t\in(t_i,t_{i+1}]$ and $s\in(t_{j-1},t_j]$ (with $j\leq i$ and $i,j\in\mathscr{I}$), and
\[
W(t,s)=[W(s,t)]^{-1}=\Phi(t,t_j)[Id+B_j]^{-1}\cdot [Id+B_i]^{-1}\Phi(t_j,s)
\]
if $t<s$, $s\in (t_j,t_{j+1}]$ and $t\in (t_{j-1},t_j]$ (with $j\leq i$ and $i,j\in\mathscr{I}$).

Let $W(t,s), t,s\in I$ denote the fundamental operator of the IDE \eqref{IL} and set $\mathscr{W}(t)=W(t,t_0)$.
Then the concept of exponential dichotomy is presented.
\begin{definition}\cite{BFS-JDE2018}
The IDEs \eqref{IL} possess an exponential dichotomy with $(P,K,\alpha)$ on $I$ if there exist a projection $P: I\to\mathscr{Z}$ and positive constants $K, \alpha$  such that
\[\begin{cases}
 \|\mathscr{W}(t) P(t)\mathscr{W}^{-1}(s)\|\leq K e^{-\alpha(t-s)}, \quad t\geq s;\\
 \|\mathscr{W}(t)(Id-P(t))\mathscr{W}^{-1}(s)\|\leq Ke^{\alpha(t-s)}, \quad t< s.
\end{cases}\]
%{\color{red}(why not putting these conditions on the integral, which is weaker? See page \pageref{NNGODE})}
\end{definition}

\begin{lemma}\label{lemma-iied}
(\cite{BFS-JDE2018} Proposition 5.21)
The IDEs \eqref{IL} possess an exponential dichotomy with  $(P,K,\alpha)$ iff the generalized ODEs
\begin{equation*}
  \frac{dz}{d\tau}=D[\Lambda(t)z], \quad t\in I,
\end{equation*}
possess an exponential dichotomy,
where $\Lambda$ is defined by \eqref{OA}.
\end{lemma}

\subsubsection{Main result}
Consider the nonlinear IDE as follows:
\begin{equation}\label{INL}
\begin{cases}
\dot{z}(t)=\widetilde{A}(t)z(t)+f(t,z(t)), \quad t\neq t_i,\\
\vartriangle z(t_i)=B_i z(t_i),\quad i\in\mathbb{Z},
\end{cases}
\end{equation}
where $f:\mathbb{R}\times \mathscr{Z}\to \mathscr{Z}$ is Perron integrable and $f(t,0)=0$.
Suppose hold:\\
(a) for all $i\in \mathbb{Z}$, there is a positive constant $C_b$ satisfying
\[
\sum\limits_{i\in\mathbb{Z}}\|B_i\|\leq C_b \quad \mathrm{and} \quad \|(Id+B_i)^{-1}\|\leq C_b;
\]
(b) there is a { Lebesgue measurable} function $m:\mathbb{R}\to\mathbb{R}$ satisfying    $\int_\mathbb{R} m(s)ds$ is finite and
\[
\left\|\int_{\mathbb{R}} \widetilde{A}(s)ds\right\| \leq \int_{\mathbb{R}} m(s)ds;
\]
%{\color{red}why globally, on all $\mathbb R$?, and not on an interval?}{\color{blue} since local integrability does not guarantee convergence of the integral as $t\to\infty$.}
(c) for any $t\in\mathbb{R}$ and $z,w\in\mathscr{Z}$,
there exists a Lebesgue measure function $\gamma:\mathbb{R}\to\mathbb{R}$ such that  $\int_\mathbb{R} \gamma(s)ds$ is finite and
\[
\|f(t,z)\|\leq \gamma(t) \quad \mathrm{and} \quad     \| f(t,z)-f(t,w)\|\leq  \gamma(t)\|z-w\|.
\]
% the nonlinear perturbation $f(t,z)$ is uniformly bounded with some $M_f>0$ and is Lipschitzian
%in $z$ with a sufficiently small Lipschitz constant $L_f$.

Now we establish the stable invariant manifold theorem for the nonlinear IDE \eqref{INL}.
\begin{theorem}
Suppose that the IDEs \eqref{IL} possess an exponential dichotomy with   $(P,K,\alpha)$. If conditions (a), (b), (c) hold,
%\[M_\gamma (1+K(1+2K))C_b^3e^{3C_b V_{\Lambda}}V_{\Lambda}^2<1,\]
then Eq. \eqref{INL} has a stable invariant manifold $\mathscr{M}_{gra}$.
\end{theorem}

\begin{proof}
Let $t, t_0\in\mathbb{R}$ and we define the Kurzweil integrable map $Q:\mathscr{Z}\times \mathbb{R}\to \mathscr{Z}$
\begin{equation}\label{QH}
  Q(z(t),t):=\int_{t_0}^t f(s,z(s))ds.
\end{equation}
For any $t\geq t_0$, the solution of Eq. \eqref{INL} with  $z(t_0)=z_0$ is given by
\[
z(t)=z_0+\int_{t_0}^{t} \widetilde{A}(s)z(s)ds+\sum\limits_{i\in\mathscr{I}_{t_0}^t}B_i z(t_i)H_{t_i}(t)+\int_{t_0}^t f(s,z(s))ds,
\]
Set $\Lambda(t)=\int_{t_0}^t \widetilde{A}(s)ds+\sum\limits_{i\in\mathscr{I}_{t_0}^{t}}B_i H_{t_i}(t)$, and by \eqref{QH}, we obtain that  for every $t\geq t_0$,
\[
z(t)=z_0+\int_{t_0}^t d[\Lambda(s)]z(s)+\int_{t_0}^t DQ(z(s),s),
\]
which is a solution of the nonlinear GODE with the initial value $z(t_0)=z_0$,
\begin{equation}\label{NNGODE}
\frac{dz}{d\tau}=D[\Lambda(t)z+Q(z,t)].
\end{equation}
Since IDEs \eqref{IL} possess an exponential dichotomy with $(P,K,\alpha)$, we derive from Lemma \ref{lemma-iied} that
\[
\frac{dz}{d\tau}=D[\Lambda(t)z]
\]
has an exponential dichotomy with the same  $(P,K,\alpha)$ on $\mathbb{R}$.

  Now we claim that $Q(z,t)$ belongs to the class $\mathscr{F}(\Omega, \mu)$, here $\Omega=\mathscr{Z}\times \mathbb{R}$.
In fact, for any $t,\widetilde{t}\in\mathbb{R}$ and $z\in\mathscr{Z}$, by condition (c) we have%there must exists a bounded nondecreasing function $\mu:\mathbb{R}\to\mathbb{R}$ such that
\[
{ \|Q(z,t)-Q(z,\widetilde{t})\|= \left\| \int_{\widetilde{t}}^t f(s,z)ds\right\|\leq \int_{\widetilde{t}}^t \gamma(s) ds }%\leq |\mu(t)-\mu(\widetilde{t})|,}
\]
and for any $t,\widetilde{t}\in\mathbb{R}$ and $z, \widetilde{z} \in\mathscr{Z}$, by also using condition (c), we have
\[
 { \|Q(z,t)-Q(z,\widetilde{t})-Q(\widetilde{z},t)+Q(\widetilde{z},\widetilde{t})\|=
 \left\| \int_{\widetilde{t}}^t \left(f(s,z)-f(s,\widetilde{z})\right)ds\right\|
% \leq \int_{\widetilde{t}}^t  \gamma(s)  \|z-\widetilde{z}\| ds \\
 \leq \|z-\widetilde{z}\| \int_{\widetilde{t}}^t \gamma(s) ds.}
\]
%{\color{red}It was sufficient to give conditions on the integral)}
    We now show that $V_\Lambda:=\mathrm{var}_c^d \Lambda<\infty$ for all $c,d\in\mathbb{R}$ and $c<d$.
    Indeed, let $D=\{t_0,t_1,\cdots,t_{|D|}\}$ be a division of $[c,d]$. Then
\[
\sum_{j=1}^{|D|} \|\Lambda(t_j)-\Lambda(t_{j-1})\|\leq
\sum_{j=1}^{|D|} \left\| \int_{t_{j-1}}^{t_j} \widetilde{A}(s)ds \right\|
+\sum_{j=1}^{|D|} \sum\limits_{i\in\mathscr{I}_{t_{j-1}}^{t_j}} \left\|B_i H_{t_i}(t) \right\|,
\]
by using condition (a) and (b), we deduce that
\[
\sum_{j=1}^{|D|} \left\| \int_{t_{j-1}}^{t_j} \widetilde{A}(s)ds \right\|
+\sum_{j=1}^{|D|} \sum\limits_{i\in\mathscr{I}_{t_{j-1}}^{t_j}} \left\|B_i H_{t_i}(t) \right\|
\leq \int_c^d m(s)ds+C_b <\infty,
\]
that is, $V_{\Lambda}<\infty$.
Set $\int_{\widetilde{t}}^t \gamma(s) ds \leq M_\gamma$ for some sufficiently small $M_\gamma>0$,
and by using conditions (a), (b) and (c), we can ensure that
$M_\gamma (1+K(1+2K))C_b^3e^{3C_b V_{\Lambda}}V_{\Lambda}^2<1.$
Consequently, all assumptions of Theorem \ref{theorem1} hold,  one can conclude that Eq. \eqref{INL} has a stable invariant manifold $\mathscr{M}_{gra}$.
This completes the proof.
\end{proof}

%\section*{Funding}
%This paper received joint support from the National Natural Science Foundation of China (No. 11671176, 11931016) and the Natural Science Foundation of Zhejiang Province (No. LZ23A010001).
%
%\section*{Data Availability Statement}
%%\hskip\parindent
%No data was used for the research in this article. It is pure mathematics.
%
%\section*{Conflict of Interest}
%\hskip\parindent
%The authors declare that they have no conflict of interest.
%
%\section*{Author Contributions}
%\hskip\parindent
% We declare that all the authors have same contributions to this paper.

%\section*{Author Emails}

\end{document}